\documentclass[11pt]{article}
\usepackage[doublespacing]{setspace}
\usepackage[left=1in,right=1in,top=1in,bottom=1in]{geometry}
\usepackage[colorlinks=true,linkcolor=RawSienna,citecolor=RawSienna,urlcolor=RawSienna,bookmarksopen=true,pdfstartview=FitB]{hyperref}
\usepackage{natbib}
\usepackage{amsmath}
\usepackage{amsfonts}
\usepackage{amssymb}
\usepackage{amsthm}
\usepackage{flexisym}
\usepackage{bbm}
\usepackage{graphicx}
\usepackage[dvipsnames]{xcolor}
\usepackage{float}
\usepackage{algorithm}
\usepackage{algpseudocode}
\usepackage{hyphenat}
\usepackage{pdfpages}
\usepackage{hyperref}
\allowdisplaybreaks
\newtheorem{theorem}{Theorem}

\newtheorem{lemma}{Lemma}

\newcommand{\be}{\begin{equation}}
\newcommand{\ee}{\end{equation}}
\newcommand{\bea}{\begin{eqnarray}}
\newcommand{\eea}{\end{eqnarray}}
\newcommand{\hs}{\hspace{.3in}}

\allowdisplaybreaks
\begin{document}

\title{\vspace{-1.4cm}A uniform FDR upper bound for a weighted FDR procedure under exchangeability}
\author{
Faith Zhang\footnote{Department of Mathematics and Statistics, Washington State University, Pullman, WA 99164, USA. $^{\ddagger}$Current address: Department of Mathematics and Statistics, University of Massachusetts, Amherst, MA 01003, USA; Email: {\tt yzhang@math.umass.edu}. $^{\dagger}$Author of correspondence. Email: {\tt xiongzhi.chen@wsu.edu}.}\ $^{\ddagger}$ \ and Xiongzhi Chen$^{\ast \dagger}$
}
\date{}
\maketitle

\begin{abstract}

For a weighted false discovery rate (FDR) procedure for multiple testing the means of equicorrelated normal random variables, we provide an analytic, non-asymptotic, uniform FDR upper bound for its FDR. Two additional and related results are also provided.

\medskip

\textit{Keywords}: Exchangeability; false discovery rate; one-way adaptive p-value weighting

\end{abstract}

\section{Introduction}
\label{secIntro}

Controlling the false discovery rate (FDR, \cite{Benjamini:1995}) has become a routine practice in multiple hypothesis testing. Recently, weighted, grouped FDR procedures (i.e., procedures that do not employ a Bayesian model, partition hypotheses into groups, are based on p-value weighting with weights constructed from data, and apply the ``BH" procedure of \cite{Benjamini:1995} to the weighted p-values, and aim at FDR control) have exemplified excellent performances due to their abilities to better adapt to the proportion of signals or incorporate potential structures among the hypotheses. Some leading such procedures have been proposed by, e.g., \cite{Storey:2004,Benjamini:2006,Blanchard:2009,Hu:2010,Guo:2020,Nandi:2018}. Among these procedures, only the ``GBH1'' procedure of \cite{Nandi:2018} has been proven to be conservative non-asymptotically under independence and the ``GS-BH" procedure of \cite{Guo:2020} under block positive dependence when hypotheses are partitioned into more than one group, whereas the rest have been proven to be conservative non-asymptotically under independence when hypotheses are partitioned into one group only. However, the ``GS-BH" procedure may not show an advantage over the BH procedure when the block size is large, and considerable numerical evidence (provided by, e.g., \cite{Benjamini:2006,Blanchard:2009,Nandi:2018}) shows that all the aforementioned weighted FDR procedures are non-asymptotically conservative under positive dependence up to a certain degree, including some cases of testing the means of equicorrelated normal random variables. Despite the support given by such numerical evidence, there does not seem to be a theoretical investigation into the observed non-asymptotic conservativeness on weighted, grouped FDR procedures under positive dependence.

Due to the special place of GBH1 among weighted, grouped FDR procedures and the fact that GBH1 reduces to Storey's procedure of \cite{Storey:2004} when there is only one group, in this note, we provide an analytic, non-asymptotic, uniform FDR upper bound for the GBH1 in the scenario of multiple testing means of equicorrelated normal random variables. The bound is not tight (mainly due to the relaxation provided by \autoref{lemma 2} to be introduced later) but quantifies the maximal FDR of the GBH1 correspondingly. As by-products, our \autoref{lemma 2} extends Lemma 3.2 of \cite{Blanchard:2008}, and \autoref{lemma 1} (to be introduced later) extends Lemma 1 of \cite{Nandi:2018}, both to the setting where p-values are not necessarily super-uniform. We remark that, for the multiple testing scenario considered here, the only existing, non-asymptotic FDR upper bounds we are aware of are provided by \cite{Reiner-Benaim:2007,Roquain:2011b}. However, these bounds are for the setting of two-sided tests, testing two hypotheses, or a simple alternative, and none of them are for weighted, grouped FDR procedures. On the other hand, FDR upper bounds for other multiple testing scenarios have been provided by, e.g., \cite{Sarkar:2002,Ferreira:2006, Finner:2007}. But these bounds are either asymptotic or are not applicable to weighted, grouped FDR procedures.
Further, our techniques can be easily modified to obtain FDR upper bounds for GBH1 and other weighted, grouped FDR procedures when they test the means of equicorrelated random variables that are members of a location-shift family. Finally, our FDR upper bound may help investigate the Tukey-Kramer conjecture that postulates that the FDR of testing the means of equi-correlated random variables may be an upper bound on the FDR of testing these means when these random variables have certain covariance structures, for which interested readers can find a brief discussion on this conjecture in \cite{Reiner-Benaim:2007}.

In the remaining part of this note, we formally state in \autoref{secModelResult} the model, multiple testing problem and our main result, and provide in \autoref{SecDependence} a detailed proof of this result.

\section{Model and main result}
\label{secModelResult}

We begin with the testing problem. Let $X_i,\  i\in \left\{0\right\}\cup \mathbb{N}_{m},$ be i.i.d.
standard normal, where $\mathbb{N}_{s}$ is defined to be the set $\left\{1,\ldots,s\right\}$ for each natural number $s$. For a constant $\rho\in(0,1)$, let $Y_i=\mu_i+\sqrt{1-\rho} X_i+\sqrt{\rho}X_0$ for $i\in \mathbb{N}_{m}$. Then $Y_i$'s are exchangeable and equicorrelated with correlation $\rho$. We simultaneously test $m$ hypotheses $H_{i}:\mu_i= 0$ versus $H_{i}^{\prime}:\mu_i>0$ for $i\in \mathbb{N}_{m}$. This scenario has been commonly used as a ``standard model'' to assess the conservativeness of an FDR procedure under dependence by, e.g., \cite{Benjamini:2006,Blanchard:2009,Finner:2007,Nandi:2018}. For each $H_{i}$, consider its associated p-value $p_i = 1- \Phi\left(Y_i\right)$, where $\Phi$ is the CDF of the standard normal distribution. The GBH1, to be applied to $\left\{p_i\right\}_{i=1}^m$, is stated as follows:

\begin{itemize}
\item \textbf{Group hypotheses}: let the $g$ non-empty sets
$\left\{  {G_{j}}\right\}_{j=1}^{g}$ be a partition of $\mathbb{N}_{m}$, and accordingly let $\left\{  H_{i}\right\}_{i=1}^{m}$ be partitioned into $\mathcal{H}_{j}=\left\{  H_{j_{k}}:k\in
G_{j}\right\}  $ for $j\in\mathbb{N}_{g}$.

\item \textbf{Construct data-adaptive weights}: fix a $\lambda\in\left(  0,1\right)  $, the tuning parameter, and for each $j$ and $G_{j}$, set
\begin{equation}
w_{j}=\frac{\left(  n_{j}-R_{j}\left(\lambda\right) +1\right)  g  }{m\left(  1-\lambda\right) },\nonumber%\label{eqe2a}
\end{equation}
where $R_{j}\left(\lambda\right)= \sum_{i\in G_{j}}\mathbf{1}{\left\{  p_{i}\leq\lambda
\right\}} $, $\mathbf{1}A$ is the indicator function of a set $A$, and $n_j = \vert G_j \vert$ is the cardinality of $G_j$.

\item \textbf{Weight p-values and reject hypotheses}: weight the p-values $p_{i}$, $i\in
G_{j}$ into $\tilde{p}_{i}=p_{i}w_{j}$, and apply the BH procedure to
$\left\{  \tilde{p}_{i}\right\}  _{i=1}^{m}$ at nominal FDR level $\alpha \in \left(0,1\right)$.

\end{itemize}

Here is our main result:

\begin{theorem} \label{theorem 1}
When $\lambda\in \left(0, 1/2\right]$ and $\rho\in \left(0, 0.34\right)$, the FDR of GBH1 is upper bounded by
\bea
B\left(\lambda,\rho,\alpha\right)&=&\alpha\left(1-\lambda\right)\left\{\frac{1}{2\sqrt{1-\rho}\ \Phi\left(\frac{1}{\sqrt{1-\rho}}\Phi^{-1}(1- \lambda)\right)}\ +\ \frac{\sqrt{2\pi}}{2}\sqrt{\frac{1-\rho}{1-2\rho}}\right.\nonumber\\
& &\hs\hs+\ \frac{\sqrt{2\pi}}{2}\left(1-\sqrt{1-\rho}\right)\sqrt{\frac{3+\sqrt{1-\rho}}{2-5\rho-\rho\sqrt{1-\rho}}}\nonumber\\
& &\hs\hs+\ \frac{\sqrt{2\pi}}{8}\left(1-\rho\right)\left(1+\sqrt{1-\rho}\right)\left(\frac{3+\sqrt{1-\rho}}{2-5\rho-\rho\sqrt{1-\rho}}\right)^{\frac{3}{2}}\nonumber\\
& &\hs\hs+\ \frac{\sqrt{\rho(1-\rho)}}{1-2\rho}\ +\  \frac{\sqrt{\rho}\left(1-\sqrt{1-\rho}\right)\left(3+\sqrt{1-\rho}\right)}{2-5\rho-\rho\sqrt{1-\rho}}\nonumber\\
& &\hs\hs+\ \left. \frac{1}{2}\sqrt{\rho}\left(1-\rho\right)\left(1+\sqrt{1-\rho}\right)\left(\frac{3+\sqrt{1-\rho}}{2-5\rho-\rho\sqrt{1-\rho}}\right)^2\right\}.\nonumber%\label{bound}
\eea
\end{theorem}

In the theorem we restrict  $\lambda\in \left(0, 0.5\right]$  mainly because researchers often choose $\lambda=\alpha$ or $\lambda=0.5$ in practice (see, e.g., \cite{Blanchard:2009} and \cite{Nandi:2018}). Also, the requirement for $\rho\in(0,0.34)$ is to ensure some integrals to be finite in the proof  of \autoref{theorem 1}, and the interval $(0, 0.34)$ is obtained by solving $\left(2-a^2\right)/\left(a^2-1\right)>0$ and $\left(5a+1-3a^3-a^2\right)/\left(\left(a^2-1\right)\left(3a+1\right)\right)>0$ resulting from the calculations for the integrals in (\ref{integrals}) (in \autoref{SecDependence}) when $a>1$ . %$\frac{2-a^2}{a^2-1}>0$ and $\frac{5a+1-3a^3-a^2}{\left(a^2-1\right)\left(3a+1\right)}>0$ given $a>1$.
The upper bound $B\left(\lambda,\rho,\alpha\right)$ is uniform in the number of groups of hypotheses, the number of hypotheses in each group, the proportions of true null hypotheses in each group and across all hypotheses, and the means of normal random variables under the alternative hypothesis.
%The ratio $B\left(\lambda,\rho,\alpha\right)/\alpha$ for $\alpha=0.05$ is partially visualized by \autoref{figure 1a}.
%\begin{figure}[t]
%\centering
%\includegraphics[height=0.3\textheight,width=0.85\textwidth]{FDRratio.pdf}
%\caption{Ratio of the FDR upper bound to the nominal FDR level $\alpha$ when $\alpha=0.05$, $\lambda\in \left(0,\ 1/2\right]$ and $\rho\in \left(0, 0.22\right)$. The curves from top to bottom are respectively associated with $\lambda$ from $0.05$ to $0.5$ with increment $0.05$. The bottom curve is for $\lambda=0.5$ and the top curve for $\lambda=0.05$.}
%\label{figure 1a}
%\end{figure}
%
%From \autoref{figure 1a}, we see that $B\left(\lambda,\rho,\alpha\right)$ is increasing in $\rho$ but decreasing in $\lambda$. Further,
The ratio $B\left(\lambda,\rho,\alpha\right)/\alpha$ is always less than $10$ when $\rho \in (0, 0.15)$, and is less than $20$ when $\rho \in (0, 0.22)$, making the upper bound $B\left(\lambda,\rho,\alpha\right)$ useful for a good range of $\rho$ when $\alpha=0.05$.
On the other hand, $\inf_{\rho\in(0, 0.34), \lambda \in (0,1/2]} B\left(\lambda,\rho,\alpha\right)=\frac{\alpha}{2}(1+\frac{\sqrt{2\pi}}{2}+\sqrt{\pi})\approx 2.0128\alpha \le 2.02 \alpha$ is achieved when $\lambda=1/2$ and $\rho=0$.
However, the case of $\rho=0$ corresponds to independence among the normal random variables and hence among the p-values, for which the FDR of GBH1 is upper bounded by $\alpha$. So, $B\left(\lambda,\rho,\alpha\right)$ is not tight, mainly because we have used the suprema of several quantities; see \autoref{result 2} and \autoref{lemma 2} and their proofs.
Nonetheless, it may be the best possible such uniform upper bound under the settings here.

%(iii) When $\rho$ is within 0.3, the ratio is still less than 100; however, the ratio increases substantially after $\rho=0.3$, which makes the upper bound useless in practice. Note that the upper bound is about twice as much as $\alpha$ at $\lambda=1/2$ and $\rho=0$.

We conduct a small simulation to assess if our FDR upper bound holds. Specifically, we create $g=5$ groups of $100$ hypotheses each, set $\mu_i=0$ when $H_i$ is a true null, independently generate $\mu_i$ from the uniform distribution on the interval $\left[0.01,3\right]$ when $H_i$ are false nulls, and compare GBH1 and the ``BH" procedure of \cite{Benjamini:1995} that does not employ the tuning parameter $\lambda$. The choice of $\mu_i$ for false nulls aims to ensure that both GBH1 and BH have low to moderate powers so that each can display reasonable FDR behaviors. The correlation $\rho$ ranges in $\{0,0.15,0.25,0.34,0.4,0.55,0.7,0.85,1\}$, i.e., from the setting of independence to that supported by our \autoref{theorem 1} and beyond, and the proportion $\pi_0$ of true null hypotheses (among the $m$ hypotheses) ranges from $0$ to $0.9$ by step size $0.1$. Both procedures are applied at the nominal FDR level $\alpha=0.05$, and the FDRs and powers of GBH1 and the BH procedures are estimated by repeating independently $1000$ times the experiment with the settings described above. Note that the BH procedure is conservative for multiple testing with one-sided p-values of equicorrelated normal random variables with variance $1$, as proved by \cite{Benjamini:2000}. The simulation results are presented by \autoref{figure 1b}. We see that the FDRs of GBH1 are upper bounded by $2 \alpha$ for $\rho \in \{0,0.15,0.25,0.34\}$ and are upper bounded by $0.45=9\alpha$ for $\rho \in \{0.4,0.55,0.7,0.85,1\}$, the range not supported by our \autoref{theorem 1}, meaning that our FDR upper bound holds but is loose. This is reasonable since our FDR upper bound is uniform in $\mu_i$ and $\pi_0$. However, we confirm again that GBH1 with $\lambda=\alpha$ has lower FDR than it does with $\lambda=0.5$ but at the cost of a bit sacrifice in power, as already checked by \cite{Nandi:2018} but only for $\rho=0.3$ and $0<\lambda\le \alpha$. Such a finding on the choice of $\lambda$ has also been reported by \cite{Benjamini:2006,Blanchard:2009} for their adaptive FDR procedures under positive dependence.  Further, GBH1 is more powerful than the BH procedure when the proportion $\pi_0$ of true null hypotheses is not very close to $1$, and when $\pi_0$ is close to $1$ and $\rho$ is relatively large, both procedures have relatively large standard deviations for their false discovery proportions. One interesting finding from the top panel in \autoref{figure 1b} is that for a given set of means $\mu_i$ when the normal random variables have perfect, positive correlation $\rho=1$, the FDRs of GBH1 and BH are not monotone increasing in $\pi_0$. This is similar to the observation reported by \cite{Reiner-Benaim:2007} on the BH procedure when it is applied to p-values of two-sided tests based on normal random variables.

\begin{figure}[t]
\centering
\includegraphics[height=0.56\textheight,width=0.85\textwidth]{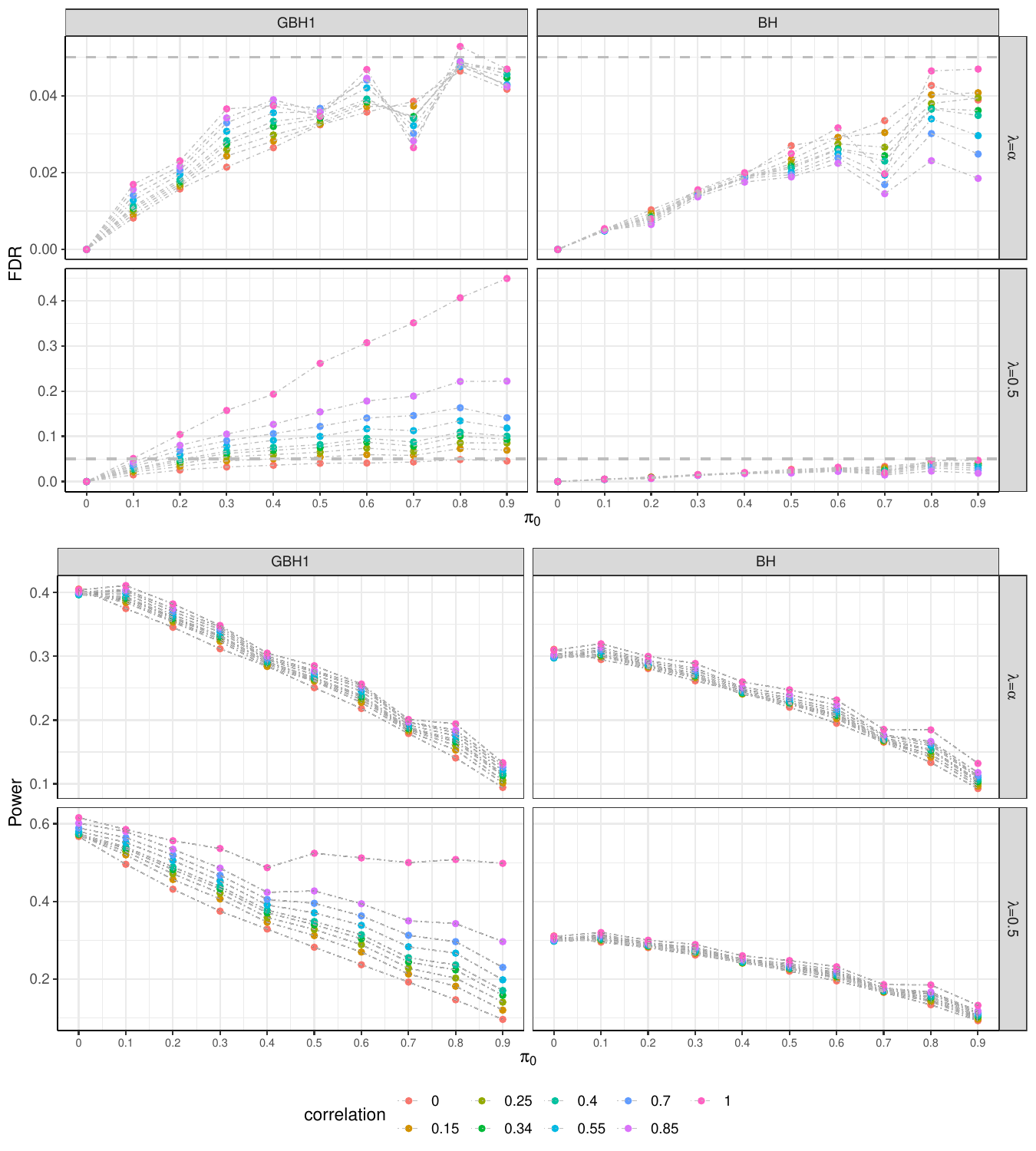}
\caption{FDRs and powers of GBH1 and the BH procedure ``BH" at nominal FDR level $\alpha=0.05$ for correlation $\rho \in \{0,0.15,0.25,0.34,0.4,0.55,0.7,0.85,1\}$ and proportion of true null hypotheses $\pi_0 \in \{0,0.1,\ldots,0.9\}$. Note that GBH1 employs tuning parameter $\lambda \in \{\alpha,0.5\}$ but BH does not.}
\label{figure 1b}
\end{figure}

\section{Proof of  \autoref{theorem 1}}
\label{SecDependence}

We provide a streamlined proof of \autoref{theorem 1} below and relegate auxiliary results to \autoref{section probability} and \autoref{section FDR}.
Let $V$ and $R$  be respectively numbers of false rejections and total rejections of GBH1, we first consider the conditional expectation $E\left(\left.V/R\ \right | X_0=x_0\right)$. Since $V/R=0$ is set when $R=0$, we can assume $R>0$ throughout the article.  Let $ I_0$ be the index set of true null hypotheses among the $m$ hypotheses, $G_{j0}=G_j\cap I_0$ the index set of true null hypotheses for group $j$, and $n_{j0}=\left|G_{j0}\right|$ the cardinality of $G_{j0}$.  Further let $\mathbf{p}= (p_1,...,p_m)$ be the vector of the $m$ p-values, and $\mathbf{p}_{-i}$  the vector obtained by excluding $p_i$ from $\mathbf{p}$.
Then
\bea E\left(\left.V/R\ \right | X_0=x_0\right)%&=&\sum\limits_{i\in I_0}E\left(\left.\mathbf{1}{\left\{\tilde{p}_{i} \leq R \frac{\alpha}{m}\right\}}/R\ \right | X_0=x_0\right)\nonumber\\
&=&\sum\limits_{j=1}^g\sum\limits_{k\in G_{j0}}E\left(\left. \mathbf{1}{\left\{p_{j_k} \leq R \frac{\alpha}{mw_j}\right\}}/R\ \right | X_0=x_0\right)\nonumber\\
&\leq&\sum\limits_{j=1}^g\sum\limits_{k\in G_{j0}}E\left(\left.\mathbf{1}{\left\{p_{j_k} \leq R \frac{\alpha}{mw_j^{(-k)}}\right\}}/R\ \right | X_0=x_0\right)  \nonumber\\ %\label{by claim 1}\\
&=&\sum\limits_{j=1}^g\sum\limits_{k\in G_{j0}}E\left[E\left(\left.\mathbf{1}{\left\{p_{j_k} \leq  R \frac{\alpha}{mw_j^{(-k)}}\right\}}/R \ \right| \ X_0=x_0,\  \mathbf{p}_{-j_k}\right)\right], \label{p}
\eea
where for each $j \in \mathbb{N}_g$ and $k \in G_j$,
%\begin{equation*}
$w_j^{(-k)}=g\left(n_j-R_{j}^{(-k)}\left(\lambda\right)\right) \left[m\left(1-\lambda\right) \right]^{-1}$ %\label{Nandi weight no k}
%\end{equation*}
with $R_j^{(-k)}\left(\lambda\right)=\sum\limits_{i\in G_j\setminus \{k\}}\mathbf{1}{ \{p_i\leq \lambda\}}$, and the inequality is due to the fact that $w_j$ is non-decreasing in $p_{j_k}$ for all $j \in \mathbb{N}_g$ , $k \in G_j$ and $w_j\geq w_j^{(-k)}$.

Define $R_j\left(\lambda,x_0\right)=\sum\limits_{i\in G_j}\mathbf{1}{ \left\{\left. p_i\leq \lambda\  \right | X_0=x_0\right\}}$
and $R_j^{(-k)}\left(\lambda,x_0\right)=\sum\limits_{i\in G_j\setminus \left\{k\right\}}\mathbf{1}{ \left\{\left. p_i\leq \lambda\  \right | X_0=x_0\right\}}$
for each $j \in \mathbb{N}_g$ and $k \in G_j$. Then the inequality (\ref{p}) implies
\bea &&E\left(\left.V/R\ \right | X_0=x_0\right)\leq \sum\limits_{j=1}^g\sum\limits_{k\in G_{j0}}E\left[\frac{\alpha M(\rho,x_0)}{mw_j^{(-k)}}\right]\label{by lemma 2}\\&=&\frac{\alpha M(\rho,x_0)}{m}\sum\limits_{j=1}^g\sum\limits_{k\in G_{j0}}E\left[\frac{m\left(1-\lambda\right)}{\left(n_j-R_{j}^{(-k)}\left(\lambda,x_0\right)\right)g}\right],\label{p2x}
\eea
where $M(\rho,x_0)$ is defined by \autoref{result 2} and the inequality (\ref{by lemma 2}) holds by \autoref{lemma 2}.
Set
$h\left(R_j\left(\lambda,x_0\right)\right)=1/g$ and
$h\left(R_j^{(-k)}\left(\lambda,x_0\right)\right)=1/g$.
Applying \autoref{lemma 1} with $h$ in place of $\tilde{h}$ to the expectation in (\ref{p2x}) gives
\bea E\left(\left.V/R\ \right | X_0=x_0\right)&\leq& \alpha\left(1-\lambda\right)M(\rho,x_0)\sum\limits_{j=1}^g\sum\limits_{k\in G_{j0}}E\left[\frac{h\left(R_j^{(-k)}\left(\lambda,x_0\right)\right) }{n_j-R_{j}^{(-k)}\left(\lambda,x_0\right)}\right]\nonumber\\&
\leq&\alpha\left(1-\lambda\right)M(\rho,x_0)\sum\limits_{j=1}^g\frac{1}{P\left(\lambda,x_0\right)}E\left[h\left(R_j\left(\lambda,x_0\right)\right)\right] \nonumber\\&
%=&\frac{\alpha\left(1-\lambda\right)M(\rho,x_0)}{P\left(\lambda,x_0\right)}\sum\limits_{j=1}^g E\left(\frac{R_j+1}{R+g}\right)\nonumber
=&\frac{\alpha\left(1-\lambda\right)M(\rho,x_0)}{P\left(\lambda,x_0\right)}\label{pot1b},
\eea
where the last equality follows from $\sum_{j=1}^g h\left(R_j\left(\lambda,x_0\right)\right)= 1$. Let $\hat{\alpha}$ be the FDR of the GBH1 procedure. Then with (\ref{pot1b}) we obtain
\bea&& \hat{\alpha}=E\left[E\left(\left.V/R\ \right | X_0\right)\right]=\int_{-\infty}^\infty E\left(\left.V/R\ \right | X_0=x_0\right)\phi(x_0)dx_0\nonumber\\
&\leq&\alpha\left(1-\lambda\right)\int_{-\infty}^\infty \frac{M(\rho,x_0)}{P\left(\lambda,x_0\right)}\phi(x_0)dx_0=\alpha\left(1-\lambda\right)\int_{-\infty}^\infty \frac{M(\rho,x_0)}{P\left(\lambda,x_0\right)}\phi\left(\frac{-b}{\sqrt{a^2-1}}\right)\left(\frac{1}{\sqrt{a^2-1}}\right)d b \nonumber\\%\label{x_0 b}\\
&\leq&\frac{\alpha\left(1-\lambda\right)}{\sqrt{a^2-1}}\int_{-\infty}^0 \left[1+\frac{4(a-1)^2+b^2}{4(a-1)} \exp{\left(\frac{b^2}{8a^2-2(a+1)^2}\right)}\right]\left(\frac{1-b}{\phi(-b)}\right)\phi\left(\frac{-b}{\sqrt{a^2-1}}\right)d b\nonumber\\
& &+\  \frac{\alpha\left(1-\lambda\right)}{\sqrt{a^2-1}}\int_{0}^\infty \frac{a}{\Phi\left(a\Phi^{-1}\left(1- \lambda\right)\right)}\phi\left(\frac{-b}{\sqrt{a^2-1}}\right)d b \label{by result 2 and 3}\\
%&\leq&\frac{\alpha\left(1-\lambda\right)}{\sqrt{a^2-1}}\int_{-\infty}^0 (1-b)\left[1+\frac{4(a-1)^2+b^2}{4(a-1)} \exp{\left(\frac{b^2}{8a^2-2(a+1)^2}\right)}\right] \exp{\left[-\frac{b^2}{2}\left(\frac{2-a^2}{a^2-1}\right)\right]}d b\nonumber\\
%& &+\   \frac{\alpha\left(1-\lambda\right)a}{\sqrt{a^2-1}\ \Phi\left(a\Phi^{-1}\left(1- \lambda\right)\right)}\int_{0}^\infty\frac{1}{\sqrt{2\pi}} \exp{\left[-\frac{b^2}{2\left(a^2-1\right)}\right]}d b\nonumber \\
&=& \frac{\alpha\left(1-\lambda\right)}{\sqrt{a^2-1}}\left(I_1+I_2+I_3+I_4+I_5+I_6\right)+\frac{\alpha\left(1-\lambda\right)a}{\sqrt{a^2-1}\ \Phi\left(a\Phi^{-1}\left(1- \lambda\right)\right)}I_7,\label{integrals}\eea
where $\phi$ denotes the density of the standard normal distribution, $x_0=\frac{-b}{\sqrt{a^2-1}}$ %in (\ref{x_0 b})
with $a=\frac{1}{\sqrt{1-\rho}}>1$, and $b=b(x_0)=-\sqrt{\frac{\rho}{1-\rho}}x_0$. Specifically, (\ref{by result 2 and 3}) is due to  \autoref{result 2}, (\ref{result 3a}) and  (\ref{result 3b}), and the $I_j$'s in (\ref{integrals}) are given in the supplementary material.
Therefore, $\hat{\alpha} \le B\left(\lambda,\rho,\alpha\right)$,
%%%%%%%%%%%%%%%%%%%%%%%%%%%%%%%%%%%%
%\bea \hat{\alpha} &\leq& \alpha\left(1-\lambda\right)\left\{\frac{a}{2\ \Phi\left(a\Phi^{-1}\left(1- \lambda\right)\right)}+\frac{\sqrt{2\pi}}{2\sqrt{2-a^2}}+(a-1) \frac{\sqrt{2\pi}}{2}\sqrt{\frac{3a+1}
%{5a+1-3a^3-a^2}}\right.\nonumber\\& &\hs\hs+\frac{\sqrt{2\pi}}{8(a-1)\sqrt{a^2-1}}\left(\frac{\left(a^2-1\right)\left(3a+1\right)}{5a+1-3a^3-a^2}\right)^{\frac{3}{2}}+\frac{\sqrt{a^2-1}}{2-a^2}
%\nonumber\\& &\hs\hs+\left.\frac{(a-1)\sqrt{a^2-1}\left(3a+1\right)}{5a+1-3a^3-a^2}+\frac{1}{2(a-1)\sqrt{a^2-1}}\left(\frac{\left(a^2-1\right)\left(3a+1\right)}{5a+1-3a^3-a^2}\right)^2\right\}\nonumber\\
%&=&B\left(\lambda,\rho,\alpha\right),\nonumber%\alpha\left(1-\lambda\right)\left\{\frac{1}{2\sqrt{1-\rho}\ \Phi\left(\frac{1}{\sqrt{1-\rho}}\Phi^{-1}(1- \lambda)\right)}\ +\ \frac{\sqrt{2\pi}}{2}\sqrt{\frac{1-\rho}{1-2\rho}}\right.\nonumber\\
%& &\hs\hs+\ \frac{\sqrt{2\pi}}{2}\left(1-\sqrt{1-\rho}\right)\sqrt{\frac{3+\sqrt{1-\rho}}{2-5\rho-\rho\sqrt{1-\rho}}}\nonumber\\
%& &\hs\hs+\ \frac{\sqrt{2\pi}}{8}\left(1-\rho\right)\left(1+\sqrt{1-\rho}\right)\left(\frac{3+\sqrt{1-\rho}}{2-5\rho-\rho\sqrt{1-\rho}}\right)^{\frac{3}{2}}\nonumber\\
%& &\hs\hs+\ \frac{\sqrt{\rho(1-\rho)}}{1-2\rho}\ +\  \frac{\sqrt{\rho}\left(1-\sqrt{1-\rho}\right)\left(3+\sqrt{1-\rho}\right)}{2-5\rho-\rho\sqrt{1-\rho}}\nonumber\\
%& &\hs\hs+\ \left.\frac{1}{2}\sqrt{\rho}\left(1-\rho\right)\left(1+\sqrt{1-\rho}\right)\left(\frac{3+\sqrt{1-\rho}}{2-5\rho-\rho\sqrt{1-\rho}}\right)^2\right\}\nonumber.
%\eea
where $B\left(\lambda,\rho,\alpha\right)$ is given in the statement of \autoref{theorem 1}.

\subsection{Bounding the probability of a conditional false rejection} \label{section probability}

For $t\in\left [0,1\right]$ and $i\in I_0$, we have the ``probability of a conditional false rejection'' as
\bea \Pr\left(\left.p_i\leq t\ \right | X_0=x_0\right)=\Pr\left[1-\Phi\left(\sqrt{1-\rho} X_i+\sqrt{\rho}x_0\right)\leq t\right]\nonumber=1-\Phi\left[\frac{\Phi^{-1}(1-t)}{\sqrt{1-\rho}}-\sqrt{\frac{\rho}{1-\rho}}x_0\right], \nonumber\eea
which induces the ratio \be g\left(t\right)=t^{-1}{\Pr\left(\left.p_i\leq t\ \right | X_0=x_0\right)}=t^{-1} \left\{1-\Phi\left[\frac{\Phi^{-1}(1-t)}{\sqrt{1-\rho}}-\sqrt{\frac{\rho}{1-\rho}}x_0\right]\right\} .\nonumber\ee %\label{ratio in t}
Note that $g\left(0\right)=0$ is set since $\lim_{t\to 0} g\left(t\right)=0$ holds, and that $g(1)=1$. The key result in this subsection is an upper bound on $g$ (or $f$ to be introduced later), given by \autoref{result 2} below.

First, let us verify that $g$ is upper bounded on $\left [0,1\right]$. Setting $a=\frac{1}{\sqrt{1-\rho}}>1$ and $b=b(x_0)=-\sqrt{\frac{\rho}{1-\rho}}x_0\in \mathbbm{R}$ gives an equivalent representation of $g$ as $ g\left(t\right)=t^{-1}\left\{1-\Phi\left[a\Phi^{-1}\left(1-t\right)+b\right]\right\}$. Clearly, $g\left(t\right)<1$ when $t\in \left[0,\ 1-\Phi\left(\frac{-b}{a-1}\right)\right)$, and $g\left(t\right)\geq 1$ when $t\in \left[1-\Phi\left(\frac{-b}{a-1}\right),\ 1\right]$. However, $g\left(t\right)$ is continuous for $t\in\left [0,1\right]$. So, $g$ attains its maximum at some $\tilde{t}%=\operatorname*{argmax}_t g\left(t\right)
\in \left[1-\Phi\left(\frac{-b}{a-1}\right),\ 1\right]$ %by the extreme value theorem
 and is thus bounded on $\left [0,1\right]$.

Secondly, let us find an upper bound for $g$. Setting $x=\Phi^{-1}(1-t)$ with $\Phi^{-1}(1)=\infty$ and $\Phi^{-1}(0)=-\infty$  gives another equivalent representation of $g$ as $ f\left(x\right)=\left[1-\Phi\left(ax+b\right)\right]/\left[1-\Phi\left(x\right)\right]$.
So, it suffices to upper bound $f$ on $\mathbbm{R}$. Clearly, $f\left(x\right)< 1$ when $x>\frac{-b}{a-1}$, $f\left(x\right)> 1$ when $x< \frac{-b}{a-1}$, $f\left(\frac{-b}{a-1}\right)= 1$, and $\lim_{x \to  -\infty} f\left(x\right)=1$. So $\operatorname*{argmax}_{x \in \mathbbm{R}} f\left(x\right) \subseteq \left(-\infty, \frac{-b}{a-1}\right)$, and it suffices to upper bound $f$ on $\left(-\infty, \frac{-b}{a-1}\right)$. To this end, we need the following:

%%% the first lemma
\begin{lemma}\label{result 1}
For $i\in I_0$, \bea\Phi\left(ax+b\right)=\Phi\left(x\right)+ \frac{1}{\sqrt{2\pi}}\left[(a-1)x+b\right]\ \exp{\left[-\frac{1}{2}\left(\frac{2ax+b}{a+1}\right)^2\right]}. \label{MVT}\eea
%or equivalently, \bea & &\Phi(\frac{1}{\sqrt{1-\rho}}x-\sqrt{\frac{\rho}{1-\rho}}x_0)\nonumber\\&=&\Phi\left(x\right)+\frac{1}{\sqrt{2\pi}}\left(\frac{1-\sqrt{1-\rho}}{\sqrt{1-\rho}}x-\sqrt{\frac{\rho}{1-\rho}}x_0\right)\exp{\left\{-\frac{1}{2}\left[\frac{2}{1+\sqrt{1-\rho}}x-\frac{\sqrt{\rho}x_0}{1+\sqrt{1-\rho}}\right]^2\right\}}\nonumber.\eea
\end{lemma}
The proof of \autoref{result 1} is given in the supplementary material. With \autoref{result 1}, we can obtain an upper bound for $f$ (or $g$) as follows:
\begin{lemma}\label{result 2}
For $\rho\in(0,\ 0.34)$ we have $f\left(x\right)\leq M(\rho,x_0)$, where %$b=-\sqrt{\rho/(1-\rho)}x_0$,
\bea M(\rho,x_0)
&=& \begin{cases}       a & b\geq 0 \\        1+\frac{4(a-1)^2+b^2}{4(a-1)} \exp{\left(\frac{b^2}{8a^2-2(a+1)^2}\right)} & b<0    \end{cases}\nonumber \\
&=&\begin{cases}  \frac{1}{\sqrt{1-\rho}} &x_0\leq 0\\ 1+\frac{4\left(1-\sqrt{1-\rho}\right)^2+\rho x_0^2}{4(\sqrt{1-\rho}-1+\rho)} \exp{\left(\frac{\rho x_0^2}{4\left(1-\sqrt{1-\rho}\right)+2\rho}\right)} & x_0>0\ .\end{cases}\nonumber\eea
\end{lemma}

\begin{proof}
We will divide the arguments for two cases.
Case (1): $b\geq 0$ (i.e., $x_0\leq 0)$. Regardless of the values of $a$ and $b$, we have $f\left(x\right)<\left(1-\Phi\left(0\right)\right)^{-1}= \left(1/2\right)^{-1}=2$ for $x\in \left(-\infty, \frac{-b}{a-1}\right)$. On the other hand,
\begin{equation*}
  f'(x)=\left[1-\Phi\left(x\right)\right]^{-2}\left\{-a\phi\left(ax+b\right)\left[1-\Phi\left(x\right)\right]+\phi\left(x\right)\left[1-\Phi\left(ax+b\right)\right]\right\},
\end{equation*}
$f'(\tilde{x})=0$ for each $\tilde{x} \in \operatorname*{argmax}_{x \in \mathbbm{R}} f\left(x\right)$, and $a\tilde{x}+b<\tilde{x}<0$. So, $f(\tilde{x})=\left[1-\Phi(a\tilde{x}+b)\right]/\left[1-\Phi(\tilde{x})\right]=a\phi(a\tilde{x}+b)/\phi(\tilde{x})<a$, %$f(\tilde{x})=\frac{1-\Phi(a\tilde{x}+b)}{1-\Phi(\tilde{x})}=\frac{a\phi(a\tilde{x}+b)}{\phi(\tilde{x})}<a$
and $f\left(x\right)<\min{\{2,a\}}=\min{\{2,\frac{1}{\sqrt{1-\rho}}\}}=\frac{1}{\sqrt{1-\rho}}=a$ when $\rho\in(0,0.34)$.

Case (2): $b<0$ (i.e., $x_0>0$). We have 2 subcases. If $\tilde{x}\in (-\infty, 0)$, then $f\left(x\right)<a$ by the same argument as above. If $\tilde{x}\in \left[0, \frac{-b}{a-1}\right)$, then \be1-\Phi\left(x\right)>\frac{2\phi\left(x\right)}{\sqrt{4+x^2}+x}\label{1942}\ee for $x\geq0$ by \cite{Birnbaum:1942}, and
\bea f\left(x\right)=\frac{1-\Phi\left(ax+b\right)}{1-\Phi\left(x\right)}&<&1-\frac{1}{\sqrt{2\pi}}\left[(a-1)x+b\right]\ \exp{\left[-\frac{1}{2}\left(\frac{2ax+b}{a+1}\right)^2\right]}\frac{\sqrt{4+x^2}+x}{2\phi\left(x\right)}\nonumber\\&=&1-\frac{1}{2}\left[(a-1)x+b\right]\left(\sqrt{4+x^2}+x\right)\ \exp{\left[\frac{1}{2}x^2-\frac{1}{2}\left(\frac{2ax+b}{a+1}\right)^2\right]}\nonumber\\&=&1+f_1(x)f_2(x), \nonumber\eea where the inequality follows from \autoref{result 1}.
Now we can easily verify that on $\left[0, \frac{-b}{a-1}\right)$, $$f_1(x)=-\frac{1}{2}\left[\left(a-1\right)x+b\right]\left(\sqrt{4+x^2}+x\right)\leq f_1\left(\frac{4(a-1)^2-b^2}{2(a-1)b}\right)=\frac{4(a-1)^2+b^2}{4(a-1)}$$ and $$f_2(x)=\exp{\left[\frac{1}{2}x^2-\frac{1}{2}\left(\frac{2ax+b}{a+1}\right)^2\right]}\leq f_2\left(\frac{2ab}{(a+1)^2-4a^2}\right)=\exp{\left(\frac{b^2}{8a^2-2(a+1)^2}\right)}.$$
So, $$f\left(x\right)< 1+\frac{4(a-1)^2+b^2}{4(a-1)} \exp{\left(\frac{b^2}{8a^2-2(a+1)^2}\right)}$$  and  \bea f\left(x\right)&\leq& \max{\left\{a,\ 1+\frac{4(a-1)^2+b^2}{4(a-1)} \exp{\left(\frac{b^2}{8a^2-2(a+1)^2}\right)}\right\}}\nonumber\\&=&\max{\left\{\frac{1}{\sqrt{1-\rho}},\ 1+\frac{4\left(1-\sqrt{1-\rho}\right)^2+\rho x_0^2}{4(\sqrt{1-\rho}-1+\rho)} \exp{\left(\frac{\rho x_0^2}{4\left(1-\sqrt{1-\rho}\right)+2\rho}\right)}\right\}}\nonumber\\&=&1+\frac{4\left(1-\sqrt{1-\rho}\right)^2+\rho x_0^2}{4(\sqrt{1-\rho}-1+\rho)} \exp{\left(\frac{\rho x_0^2}{4\left(1-\sqrt{1-\rho}\right)+2\rho}\right)}.\nonumber\eea
\end{proof}

\subsection{Bounding the conditional expectation involving the number of rejections}\label{section FDR}

In this subsection, we present two results that are related to conditional expectations involving the number of rejections. Write the number $R$ of rejections of the GBH1 procedure as $R\left(p_{i}, \mathbf{p}_{-i}\right)$ for each $i \in \mathbb{N}_m$.

\begin{lemma}\label{lemma 2}
For $c>0$, $j \in \mathbb{N}_g$ , $k \in G_{j0}$ and with $M(\rho,x_0)$ defined by \autoref{result 2}, \be E\left(\left.\frac{\mathbf{1}{\left\{p_{j_k} \leq c R \left(p_{j_k}, \mathbf{p}_{-{j_k}}\right)\right\}}}{R\left(p_{j_k}, \mathbf{p}_{-{j_k}}\right)} \ \right| \  X_0=x_0,\  \mathbf{p}_{-{j_k}}\right)\leq cM(\rho,x_0).\nonumber\ee
\end{lemma}

\begin{proof}
We will write the number of rejections conditional on $\mathbf{p}_{-{j_k}}$, i.e., $\left. R\left(p_{j_k}, \mathbf{p}_{-j_k}\right) \ \right| \  \mathbf{p}_{-j_k}$, as $R\left(p_{j_k}\right)$ for simplicity. Since $w_j$ non-decreases with $p_{j_k}$, $R\left(p_{j_k}\right)$ is non-increasing in $p_{j_k}$.

Let $X(x_0)=\left\{p_{j_k}:  \left.p_{j_k} \leq c R (p_{j_k})\ \right | X_0=x_0\right\}$ and $Y(x_0)=\left\{R\left(p_{j_k}\right) : p_{j_k} \in X(x_0) \right\}$. %=\left\{R\left(p_{j_k}\right) : p_{j_k} \in X(x_0) \right\}$ since obviously $R\left(p_{j_k}\right)\neq 0$ when $p_{j_k} \in X(x_0)$,
Then %$\underset{p_{j_k} \in X}{sup}p_{j_k}$
$\sup{X(x_0)}$ and $\inf{Y(x_0)}$ exist, which are denoted by $X^*(x_0)$ and $Y_*(x_0)$ %($Y_*(x_0)\geq 1$)
 respectively. Clearly, $X^*(x_0)\leq cY_*(x_0)$.
 So, \bea &&E\left(\left.\frac{\mathbf{1}{\left\{p_{j_k} \leq c R (p_{j_k}, \mathbf{p}_{-{j_k}})\right\}}}{R(p_{j_k}, \mathbf{p}_{-{j_k}}) } \ \right| \ X_0=x_0,\  \mathbf{p}_{-{j_k}}\right) \leq \frac{\Pr\left(p_{j_k} \in X(x_0) \right)}{Y_*(x_0)}\leq \frac{\Pr\left(p_{j_k} \leq X^*(x_0) \right)}{Y_*(x_0)}\nonumber\\&\leq& \frac{c\Pr\left(p_{j_k} \leq X^*(x_0) \right)}{X^*(x_0)}\nonumber\leq c
\sup_{t\in[0, 1]}\left\{\frac{\Pr(\left.p_{j_k}\leq t\ \right | X_0=x_0)}{t}\right\}\nonumber\leq c M(\rho,x_0)\ ,\nonumber\eea where the last inequality follows from \autoref{result 2}.
\end{proof}

Since conditional on $X=x_0$, the p-value $p_{j_k}$ is not necessarily super-uniform when $\rho \ne 0$, \autoref{lemma 2} extends Lemma 3.2 of \cite{Blanchard:2008}, the latter of which in our notations has $M(\rho,x_0)=1$ when the p-values $\left\{p_i\right\}_{i=1}^m$ are independent and super-uniform. However, employing the supremum $M(\rho,x_0)$ is the key reason why the upper bound provided by \autoref{theorem 1} can be loose.

\begin{lemma} \label{lemma 1}
For any non-negative, real-valued, measurable function $\tilde{h}$,
\be \sum\limits_{k\in G_{j0}} E\left[\frac{\tilde{h}\left(R_j^{(-k)}\left(\lambda,x_0\right)\right)}{n_j-R_j^{(-k)}\left(\lambda,x_0\right)}\right]\leq \frac{1}{P\left(\lambda,x_0\right)}E\left[\tilde{h}\left(R_j\left(\lambda,x_0\right)\right)\right]\nonumber, \ee where $P\left(\lambda,x_0\right)=\Pr\left(\left.p_{j_k}> \lambda\ \right | X_0=x_0\right) $ is such that
$P\left(\lambda,x_0\right)\geq \Phi\left(a\Phi^{-1}\left(1- \lambda\right)\right)$ for $b\geq 0$ and
$P\left(\lambda,x_0\right)>\frac{\phi(-b)}{1-b}$ for $b< 0$ when $\lambda\in \left(0,1/2\right]$, $j \in \mathbb{N}_g$ and $k \in G_{j0}$.
\end{lemma}

\begin{proof} Define $V_j \left(\lambda,x_0\right)=\sum\limits_{i\in G_{j0}}\mathbf{1}{ \left\{\left. p_i\leq \lambda\  \right | X_0=x_0\right\}}$. By simple algebra,
\bea &&\sum\limits_{k\in G_{j0}} E\left[\frac{\tilde{h}\left(R_j^{(-k)}\left(\lambda,x_0\right)\right)}{n_j-R_j^{(-k)}\left(\lambda,x_0\right)}\right]\nonumber\\&=&\sum\limits_{r=0}^{n_j-1}\sum\limits_{k\in G_{j0}} E\left[\frac{\mathbf{1}{\left\{R_j^{(-k)}\left(\lambda,x_0\right)=r\right\}}\tilde{h}\left(R_j^{(-k)}\left(\lambda,x_0\right)\right)}{n_j-R_j^{(-k)}\left(\lambda,x_0\right)}\right]\nonumber
\\&=&\frac{1}{P\left(\lambda,x_0\right)}\sum\limits_{r=0}^{n_j-1}\sum\limits_{k\in G_{j0}} E\left[\frac{\mathbf{1}{\left\{R_j^{(-k)}\left(\lambda,x_0\right)=r\right\}}\mathbf{1}{\left\{\left.p_{j_k}> \lambda\ \right | X_0=x_0\right\}}\tilde{h}\left(r\right)}{n_j-r}\right]\label{pol1.1}
%\\&=&\frac{1}{P\left(\lambda,x_0\right)}\sum\limits_{r=0}^{n_j-1}\sum\limits_{k\in G_{j0}} E\left[\frac{\mathbf{1}{\left\{R_j\left(\lambda,x_0\right)=r\right\}}\mathbf{1}{\left\{\left.p_{j_k}> \lambda\ \right | X_0=x_0\right\}}\tilde{h}\left(R_j\left(\lambda,x_0\right)\right)}{n_j-R_j\left(\lambda,x_0\right)}\right]\nonumber\\&=&\frac{1}{P\left(\lambda,x_0\right)}\sum\limits_{k\in G_{j0}} E\left[\frac{\mathbf{1}{\left\{\left.p_{j_k}> \lambda\ \right | X_0=x_0\right\}}\tilde{h}\left(R_j\left(\lambda,x_0\right)\right)}{n_j-R_j\left(\lambda,x_0\right)}\right]\nonumber
\\&=&\frac{1}{P\left(\lambda,x_0\right)} E\left[\frac{\left(n_{j0}-V_j\left(\lambda,x_0\right) \right)\tilde{h}\left(R_j\left(\lambda,x_0\right)\right)}{n_j-R_j\left(\lambda,x_0\right)}\right]\nonumber\leq \frac{1}{P\left(\lambda,x_0\right)}E\left[\tilde{h}\left(R_j\left(\lambda,x_0\right)\right)\right], \nonumber %\label{pol1.2}
\eea
where (\ref{pol1.1}) is due to the independence among the p-values conditioned on $X=x_0$ and the inequality is due to the fact that $n_{j0}-V_j \left(\lambda,x_0\right) \leq n_j-R_j\left(\lambda,x_0\right)$.

Now consider $\tilde{P}\left(\lambda,x_0\right)=\Pr(\left.p_i> \lambda\ \right | X_0=x_0)$
for $\lambda\in \left(0,1/2\right]$ and $i\in I_0$. Then
\be \tilde{P}\left(\lambda,x_0\right)=\Pr\left(1-\Phi\left(\sqrt{1-\rho} X_i+\sqrt{\rho}x_0\right)> \lambda\right)=\Phi\left(a\Phi^{-1}\left(1- \lambda\right)+b\right).\nonumber\ee %\label{plx}
Clearly, \be \tilde{P}\left(\lambda,x_0\right)\geq \Phi\left(a\Phi^{-1}\left(1- \lambda\right)\right) \quad \text{when} \quad b\geq 0,\label{result 3a}\ee whereas when $b< 0$, \bea \tilde{P}\left(\lambda,x_0\right)\geq \Phi(b)=1-\Phi(-b)>\frac{2\phi(-b)}{\sqrt{4+b^2}-b}>\frac{\phi(-b)}{1-b},\label{result 3b}\eea
where we have applied (\ref{1942}) to obtain the second inequality. When $k \in G_{j0}$, $j_k$ has to be equal to some $i^{\prime} \in I_0$, which justifies the claim on $P\left(\lambda,x_0\right)$.
\end{proof}

\autoref{lemma 1} extends Lemma 1 of \cite{Nandi:2018}, in that, when $\rho=0$, i.e., when the p-values are independent, $P\left(\lambda,x_0\right)=1-\lambda$ holds for the former and hence reduces to the latter.

\bibliographystyle{dcu}
\bibliography{fdr}

\end{document}